\documentclass[12pt,twoside]{amsart}
\usepackage{amssymb,amsmath,amsthm, amscd, enumerate, mathrsfs}
\usepackage{graphicx, hhline}
\usepackage[all]{xy}
\usepackage{hyperref}
\hypersetup{colorlinks=true}

\title[Corrigendum]
{Corrigendum:~On subadditivity of the logarithmic Kodaira dimension}
\author{Osamu Fujino} 
\date{2019/12/21, version 0.06}
\address{Department of Mathematics, Graduate School of 
Science, Osaka University, Toyonaka, Osaka 560-0043, Japan} 
\email{fujino@math.sci.osaka-u.ac.jp} 

\newcommand{\rank}[0]{{\operatorname{rank}}}
\newcommand{\Supp}[0]{{\operatorname{Supp}}}
\newtheorem{thm}{Theorem}[section]
\newtheorem{lem}[thm]{Lemma}

\newtheorem{conj}[thm]{Conjecture}

\theoremstyle{definition}

\newtheorem{rem}[thm]{Remark}

\newtheorem*{ack}{Acknowledgments} 
\newtheorem{say}[thm]{}

\makeatletter
    
    \@addtoreset{equation}{section}
\makeatother

\begin{document}

\maketitle 

\begin{abstract}
John Lesieutre constructed an example 
satisfying $\kappa_\sigma\ne \kappa_\nu$. 
This says that the proof of the inequalities 
in Theorems 1.3, 1.9, and Remark 3.8 
in [O.~Fujino, 
On subadditivity of the logarithmic Kodaira dimension, 
J. Math. Soc. Japan \textbf{69} (2017), no. 4, 1565--1581] 
is insufficient. 
We claim that some weaker inequalities still hold true and they are 
sufficient for various applications. 
\end{abstract}

\tableofcontents 

\section{Introduction} 

In \cite{lesieutre}, John Lesieutre constructs a smooth projective 
threefold $X$ and a pseudo-effective 
$\mathbb R$-divisor $D$ on $X$ such that 
$\kappa_\sigma(D)=1$ and $\kappa _\nu(D)=2$. 
This means that the equality $\kappa _\sigma=\kappa_\nu$ does not 
always hold true. In the proof of \cite[Theorem 1.3]{fujino}, 
we used the following lemma (see \cite[Lemma 2.8]{fujino}), which 
is a special 
case of \cite[Theorem 6.7 (7)]{lehmann}. 

\begin{lem}\label{f-lem1.1}
Let $D$ be a pseudo-effective 
Cartier divisor on a smooth 
projective variety $X$. 
We fix some sufficiently ample Cartier divisor 
$A$ on $X$. 
Then there exist positive constants $C_1$ and $C_2$ 
such that 
\begin{equation*}
C_1m^{\kappa_{\sigma}(X, D)}\leq 
\dim H^0(X, \mathcal O_X(mD+A))\leq 
C_2m^{\kappa _{\sigma}(X, D)}
\end{equation*} 
for every sufficiently large $m$. 
\end{lem}

Unfortunately, the proof of \cite[Theorem 6.7 (7)]{lehmann} 
in \cite{lehmann} (see also \cite{eckl}) depends on the wrong fact 
that $\kappa_\sigma=\kappa_\nu$ always holds. 
Moreover, Lesieutre's example says that 
\cite[Theorem 6.7 (7)]{lehmann} is not true 
when $D$ is an $\mathbb R$-divisor. 
Therefore, this trouble damages \cite[Theorems 1.3, 1.9, and Remark 
3.8]{fujino}. 
This means that the proof of 
the inequalities in \cite[Theorems 1.3 and 1.9]{fujino} 
and \cite[Chapter V, 4.1.~Theorem (1)]{nakayama} is incomplete. 

In this paper, we explain that slightly weaker inequalities 
than the original ones in \cite[Theorems 1.3 and 1.9]{fujino} 
and \cite[Chapter V, 4.1.~Theorem (1)]{nakayama} still 
hold true. Fortunately, these weaker inequalities 
are sufficient for \cite[Corollaries 1.5 and 1.6]{fujino} 
and some other applications. Note that one of the 
main purposes of \cite{fujino} is to reduce Iitaka's 
subadditivity conjecture on the 
logarithmic Kodaira dimension $\overline \kappa$ 
(see \cite[Conjecture 1.1]{fujino}) to 
a special case of the generalized abundance 
conjecture (see \cite[Conjecture 1.4]{fujino}). 
For that purpose, one of the weaker inequalities in 
Theorem \ref{q-thm2.1} below is sufficient. 

\begin{rem}\label{f-rem1.2}
Kenta Hashizume reduces Iitaka's subadditivity conjecture 
on the logarithmic Kodaira dimension $\overline \kappa$ 
to the generalized abundance conjecture for sufficiently general 
fibers (see \cite[Theorem 1.2]{hashizume}). In some sense, 
his result is better than the one in \cite{fujino}. 
We note that 
his proof uses \cite[Theorem 4.3]{gongyo-lehmann} 
(see Theorem \ref{thm3.2} below) and that the proof of 
\cite[Theorem 4.3]{gongyo-lehmann} uses 
\cite[Chapter V, 4.2.~Corollary]{nakayama} 
which follows from \cite[Chapter V, 4.1.~Theorem (1)]{nakayama}. 
Fortunately, the inequality \eqref{eq3.3} below, which is 
weaker than the one in \cite[Chapter V, 4.1.~Theorem (1)]{nakayama}, 
is sufficient for our purpose. So there are no troubles in \cite{hashizume}. 
\end{rem}

\begin{rem}\label{f-rem1.3}
We note that \cite[Lemma 2.4.9]{fujino-foundations} 
is nothing but \cite[Theorem 6.7 (7)]{lehmann}. 
Fortunately, however, we do not use it directly in \cite{fujino-foundations}. 
\end{rem}

It is highly desirable to solve 
the following conjecture. 

\begin{conj}\label{f-conj1.4}
Let $X$ be a smooth projective variety and let $D$ be a 
pseudo-effective $\mathbb R$-divisor 
on $X$. 
Then there exist a positive integer $m_0$, a 
positive rational number $C$, and an ample Cartier divisor 
$A$ on $X$ such that 
\begin{equation}\label{p-eq1.1}
Cm^{\kappa_\sigma(X, D)}\leq \dim H^0(X, \mathcal O_X
(\lfloor mm_0D\rfloor +A))
\end{equation}
holds for every large positive integer $m$. 
\end{conj}

If Conjecture \ref{f-conj1.4} is true, then 
there are no troubles in \cite[Theorems 1.3 and 1.9]{fujino} 
and \cite[Chapter V, 4.1.~Theorem (1)]{nakayama}. 

\medskip 

The following observation may help the reader understand this paper, 
the trouble in \cite[Chapter V, 4.1.~Theorem (1)]{nakayama}, 
and Conjecture \ref{f-conj1.4}. 

\begin{say}[Observation]\label{f-say1.5}
Let us consider 
$$
f, g: \mathbb Z_{\geq 0}\longrightarrow \mathbb Z_{\geq 0}
$$ such that 
\begin{equation}\label{p-eq1.2}
\limsup _{m\to \infty}
f(m)>0 \quad \quad \text{and} 
\quad \quad 
\limsup_{m\to \infty} g(m)>0. 
\end{equation}
We want to prove 
\begin{equation}\label{p-eq1.3}
\limsup_{m\to \infty} \left(f(m)g(m)\right)>0. 
\end{equation}
In general, \eqref{p-eq1.3} does not follow from 
\eqref{p-eq1.2}. 
It may happen that $f(m)g(m)=0$ holds for every 
$m$. 
If there exists a positive constant $C$ such that 
$f(m)\geq C$ for every large positive integer $m$, 
then we have 
$$
\limsup_{m\to \infty} \left(f(m)g(m)\right)>0.
$$
\end{say}

\begin{ack} The 
author was partially supported by 
JSPS KAKENHI Grant Numbers JP16H03925, JP16H06337.
He thanks Professors Noboru 
Nakayama, Thomas Eckl, Brian Lehmann, John Lesieutre, 
Yoshinori Gongyo, Kenta Hashizume, Haidong Liu, Sung Rak Choi, and 
Jinhyung Park. 
\end{ack}

We will freely use the notation in \cite{fujino} and 
work over $\mathbb C$, the complex number 
field, throughout this paper.  

\section{On \cite[Theorems 1.3 and 1.9]{fujino}}
The proof of the main theorem of \cite{fujino}, that is, 
\cite[Theorem 1.3]{fujino}, is incomplete. 
Here we will prove slightly weaker inequalities. 

\begin{thm}[{see \cite[Theorem 1.3]{fujino}}]
\label{q-thm2.1}
Let $f:X\to Y$ be a surjective morphism between smooth projective 
varieties with connected fibers. 
Let $D_X$ $($resp.~$D_Y$$)$ be a simple normal crossing 
divisor on $X$ $($resp.~$Y$$)$. 
Assume that $\Supp f^*D_Y\subset \Supp D_X$. 
Then we have 
$$
\kappa _{\sigma}(X, K_X+D_X)\geq \kappa _{\sigma} 
(F, K_F+D_X|_F)+\kappa(Y, K_Y+D_Y)
$$ 
and 
$$
\kappa _{\sigma}(X, K_X+D_X)\geq \kappa
(F, K_F+D_X|_F)+\kappa_{\sigma}(Y, K_Y+D_Y), 
$$
where $F$ is a sufficiently general fiber of $f:X\to Y$. 
\end{thm}

The inequalities in 
Theorem \ref{q-thm2.1} follow from the proof of 
\cite[Theorem 1.3]{fujino} without any difficulties. 

\medskip 

Before we prove Theorem \ref{q-thm2.1}, 
we give a small remark on $\kappa_\sigma$. 

\begin{rem}\label{q-rem2.2}
Let $D$ be a pseudo-effective $\mathbb R$-divisor 
on a smooth projective variety $X$. 
Then $\kappa _\sigma(X, D)=\kappa _\sigma(X, lD)$ holds 
for every positive integer $l$. 
We note that 
$\kappa_\sigma(X, lD)\geq \kappa_\sigma(X, D)$ holds 
by \cite[Chapter V, 2.7.~Proposition (1)]{nakayama} since 
$lD-D=(l-1)D$ is pseudo-effective. 
By definition, $\kappa_\sigma(X, lD)\leq \kappa_\sigma(X, D)$ always 
holds. 
\end{rem}

Let us prove Theorem \ref{q-thm2.1}. 

\begin{proof}[Proof of Theorem \ref{q-thm2.1}]
In this proof, we will freely use the notation in the proof of 
\cite[Theorem 1.3]{fujino}. 
In the proof of \cite[Theorem 1.3]{fujino}, 
we have 
\begin{equation}\label{eq1.1}
\begin{split}
&\dim H^0(X, \mathcal O_X(mk(K_X+D_X)+A+2f^*H))\\
&\geq r(mD; A)\cdot \dim H^0(Y, \mathcal O_Y(mk(K_Y+D_Y)+H))
\end{split}
\end{equation} 
for every positive integer $m$, where 
\begin{equation}\label{eq1.2}
D=k(K_{X/Y}+D_X-f^*D_Y) 
\end{equation} 
and 
\begin{equation}\label{eq1.3}
r(mD; A)=\rank f_*\mathcal O_X(mD+A). 
\end{equation}
We can take a positive integer $m_0$ and a 
positive real number $C_0$ such that 
\begin{equation}\label{eq1.4}
C_0m^{\kappa (F, D|_F)}\leq r(mm_0D; A)
\end{equation}
for every large positive integer $m$. 
Since $\kappa (F, D|_F)=\kappa (F, K_F+D_X|_F)$, we have 
\begin{equation}\label{eq1.5}
\begin{split}
&\dim H^0(X, \mathcal O_X(mm_0k(K_X+D_X)+A+2f^*H))\\
&\geq C_0m^{\kappa(F, K_F+D_X|_F)}
\cdot \dim H^0(Y, \mathcal O_Y(mm_0k(K_Y+D_Y)+H)) 
\end{split}
\end{equation} 
for every positive integer $m$ by \eqref{eq1.1} and 
\eqref{eq1.4}. 
We may assume that $H$ is sufficiently ample. 
Then we get 
\begin{equation}\label{eq1.6}
\limsup_{m\to \infty} \frac{\dim H^0(X, 
\mathcal O_X(mm_0k(K_X+D_X)+A+2f^*H))}
{m^{\kappa (F, K_F+D_X|_F)+\kappa _{\sigma}(Y, K_Y+D_Y)}}>0
\end{equation}
by \eqref{eq1.5} and the definition of $\kappa _\sigma(Y, K_Y+D_Y)$. 
This means that the following inequality 
\begin{equation}\label{eq1.7}
\kappa _\sigma(X, K_X+D_X)\geq \kappa (F, K_F+D_X|_F)+
\kappa _\sigma(Y, K_Y+D_Y)
\end{equation} 
holds. 

Similarly, we can take a positive integer $m_1$ and 
a positive real number $C_1$ such that 
\begin{equation}\label{eq1.8}
\begin{split}
C_1m^{\kappa (Y, K_Y+D_Y)}&\leq 
\dim H^0(Y, \mathcal O_Y(mm_1k(K_Y+D_Y)))\\ 
&\leq \dim H^0(Y, \mathcal O_Y(mm_1k(K_Y+D_Y)+H))
\end{split}
\end{equation}
for every large positive integer $m$ by the definition 
of $\kappa (Y, K_Y+D_Y)$ if 
$H$ is a sufficiently ample Cartier divisor. 
Then, by \eqref{eq1.1} and \eqref{eq1.8}, we have 
\begin{equation}\label{eq1.9}
\begin{split}
&\dim H^0(X, \mathcal O_X(mm_1k(K_X+D_X)+A+2f^*H))\\
&\geq C_1m^{\kappa(Y, K_Y+D_Y)}
\cdot r(mm_1D; A)
\end{split}
\end{equation}
for every large positive integer $m$. 
Therefore, we get 
\begin{equation}\label{eq1.10}
\limsup_{m\to \infty} \frac{\dim H^0(X, 
\mathcal O_X(mm_1k(K_X+D_X)+A+2f^*H))}
{m^{\kappa _\sigma(F, K_F+D_X|_F)+\kappa(Y, K_Y+D_Y)}}>0 
\end{equation}
when $A$ is sufficiently ample. 
Note that 
\begin{equation}\label{eq1.11}
\sigma (m_1D|_F; A|_F)=\max \left\{
k\in \mathbb Z_{\geq 0} \cup \{-\infty\}\, 
\left|\, \underset{m\to \infty}{\limsup}\frac{r(mm_1D; A)}{m^k}>0\right.\right\} 
\end{equation}
for a sufficiently general fiber $F$ of $f:X\to Y$ and 
that 
\begin{equation}\label{eq1.12}
\begin{split}
\kappa _\sigma (F, K_F+D_X|_F)
&=\kappa_\sigma(F, D|_F)
\\&=\kappa_\sigma (F, m_1D|_F)
\\&=\max\{\sigma(m_1D|_F; A|_F)\, |\, {\text{$A$ is very ample}}\}. 
\end{split}
\end{equation}
Hence we have the inequality 
\begin{equation}\label{eq1.13}
\kappa _\sigma(X, K_X+D_X)\geq \kappa_\sigma (F, K_F+D_X|_F)+
\kappa (Y, K_Y+D_Y) 
\end{equation} 
by \eqref{eq1.10}. 
\end{proof}

Of course, we have to weaken inequalities 
in \cite[Theorem 4.12.1 and Corollary 4.12.2]{fujino-foundations} following 
Theorem \ref{q-thm2.1}. 

\medskip 

Theorem 1.9 in \cite{fujino} has the same trouble as \cite[Theorem 1.3]{fujino}. 
Of course, we can prove slightly weaker inequalities. 

\begin{thm}[{see \cite[Theorem 1.9]{fujino}}]\label{q-thm2.3}
Let $f:X\to Y$ be a proper surjective morphism 
from a normal variety $X$ onto a smooth 
complete variety $Y$ with connected fibers. 
Let $D_X$ be an effective $\mathbb Q$-divisor 
on $X$ such that $(X, D_X)$ is lc 
and let $D_Y$ be a simple normal crossing divisor on $Y$. 
Assume that $\Supp f^*D_Y\subset \lfloor D_X\rfloor$, 
where $\lfloor D_X\rfloor$ is the round-down of $D_X$. 
Then we have 
$$
\kappa _\sigma (X, K_X+D_X)\geq 
\kappa _\sigma(F, K_F+D_X|_F)+\kappa
(Y, K_Y+D_Y) 
$$ 
and 
$$
\kappa _\sigma (X, K_X+D_X)\geq 
\kappa(F, K_F+D_X|_F)+\kappa _\sigma 
(Y, K_Y+D_Y),  
$$ 
where $F$ is a sufficiently general fiber of $f:X\to Y$. 
\end{thm}

It is obvious how to modify the proof of 
\cite[Theorem 1.9]{fujino} in order to get Theorem \ref{q-thm2.3}. 
For the details, see the proof of Theorem \ref{q-thm2.1} above. 

\medskip 

As we said above, the inequalities in Theorem \ref{q-thm2.1} 
are sufficient for \cite[Corollaries 1.5 and 1.6]{fujino}. 
The reader can check it without any difficulties.

\section{On \cite[Remark 3.8]{fujino}, 
\cite[Chapter V, 4.1.~Theorem (1)]{nakayama}, and so on}

In \cite[Remark 3.8]{fujino}, the author pointed out a gap in 
the proof of \cite[Chapter V, 4.1.~Theorem (1)]{nakayama} 
and filled it by using \cite[Theorem 6.7 (7)]{lehmann} (see also 
\cite[Remark 4.11.6]{fujino-foundations}). 
Therefore, the proof of the inequalities in 
\cite[Chapter V, 4.1.~Theorem (1)]{nakayama} is still incomplete. 

\begin{say}[Nakayama's inequality for $\kappa_\sigma$]\label{say3.1} 
Here, we will freely use the notation in \cite[Chapter V, 4.1.~Theorem]
{nakayama}. 
In \cite[Chapter V, 4.1.~Theorem (1)]{nakayama}, Nakayama claims 
that the inequality 
\begin{equation}\label{eq3.1}
\kappa_\sigma(D+f^*Q)\geq \kappa _\sigma(D; X/Y)
+\kappa_\sigma(Q)
\end{equation}
holds. Unfortunately, this inequality \eqref{eq3.1} 
does not follow directly from the inequality 
\begin{equation}\label{eq3.2}
h^0(X, \lceil m(D+f^*Q)\rceil +A+2f^*H)
\geq r(mD; A) \cdot h^0(Y, \lfloor mQ\rfloor +H)
\end{equation}
established in the proof of 
\cite[Chapter V, 4.1.~Theorem (1)]{nakayama}. 
By the same argument as in the proof of 
Theorem \ref{q-thm2.1} above, 
by using \cite[Chapter II, 3.7.~Theorem]{nakayama}, 
we can prove 
\begin{equation}\label{eq3.3}
\kappa_\sigma(D+f^*Q)\geq \kappa _\sigma(D; X/Y)
+\kappa(Q)
\end{equation}
and 
\begin{equation}\label{eq3.4}
\kappa_\sigma(D+f^*Q)\geq \kappa (D; X/Y)
+\kappa_\sigma(Q). 
\end{equation}  
If we put $D=K_{X/Y}+\Delta$ and $Q=K_Y$, then 
we have 
\begin{equation}\label{eq3.5}
\kappa_\sigma(K_X+\Delta)\geq \kappa _\sigma(K_{X_y}+\Delta|_{X_y})
+\kappa(K_Y)
\end{equation}
and 
\begin{equation}\label{eq3.6}
\kappa_\sigma(K_X+\Delta)\geq \kappa (K_{X_y}+\Delta|_{X_y})
+\kappa_\sigma(K_Y). 
\end{equation}  
Lesieutre's example does not affect \cite[Chapter V, 4.1.~Theorem (2)]
{nakayama} because we do not use $\kappa_\sigma$ for the proof of 
\cite[Chapter V, 4.1.~Theorem (2)]{nakayama}. 
\end{say}

We note that the inequalities \eqref{eq3.3} and 
\eqref{eq3.4} are sufficient for the proof of 
\cite[Theorem 4.12.8]{fujino-foundations}. 

\medskip 

The inequality \eqref{eq3.1} has already played an important 
role in the theory of minimal models. 
The following result is very well known and has already been 
used in various papers. 

\begin{thm}[{\cite[Remark 2.6]{dhp} and 
\cite[Theorem 4.3]{gongyo-lehmann}}]\label{thm3.2}
Let $(X, \Delta)$ be a projective klt pair such that $\Delta$ is 
a $\mathbb Q$-divisor. 
Then $(X, \Delta)$ has a good minimal model 
if and only if $\kappa_\sigma(X, K_X+\Delta)=\kappa(X, K_X+\Delta)$. 
\end{thm}

In the proof of Theorem \ref{thm3.2}, 
the inequality \eqref{eq3.1} 
plays an important role in \cite[Remark 2.6]{dhp}.   
Fortunately, in \cite[Remark 2.6]{dhp}, 
the inequality \eqref{eq3.3} is sufficient because 
we need the inequality \eqref{eq3.1} in the 
case where $Q$ is a big divisor. 
In \cite[Theorem 4.3]{gongyo-lehmann}, 
Gongyo and Lehmann need \cite[Chapter V, 4.2.~Corollary]
{nakayama} in the proof of \cite[Theorem 4.3]{gongyo-lehmann}. 
We note that Nakayama uses the inequality \eqref{eq3.1} 
in the proof of \cite[Chapter V, 4.2.~Corollary]{nakayama}. 
Fortunately, we can easily see that 
\cite[Chapter V, 4.2.~Corollary]{nakayama} 
holds true because the inequality \eqref{eq3.3} is 
sufficient for that proof. We strongly recommend the reader to 
see \cite[Subsection 2.2]{hashizume-hu} for 
some related topics. 
We can find a generalization of Theorem \ref{thm3.2} 
(see \cite[Lemma 2.13]{hashizume-hu}). 

\begin{rem}\label{rem3.3}
In a recent preprint \cite{fujino-omega}, 
we introduce the notion of mixed-$\omega$-sheaves and 
discuss some topics related to Nakayama's inequality 
for $\kappa_\sigma$. 
For the details, see \cite[Section 11]{fujino-omega}. 
\end{rem}


\end{document}